\documentclass{amsart}
\usepackage{amsmath}%
\usepackage{amsfonts}%
\usepackage{amssymb}%
\usepackage{graphicx}
\newtheorem{theorem}{Theorem}
\theoremstyle{plain}

\newtheorem{corollary}{Corollary}

\newtheorem{definition}{Definition}

\newtheorem{proposition}{Proposition}
\newtheorem{remark}{Remark}

\numberwithin{equation}{section}
\begin{document}
\title[Notes on non-commutative ergodic theorems]{Notes on ergodic theorems in non-commutative symmetric spaces}
%\author{Chilin}
\author{Genady Ya. Grabarnik}
\address[GYaG]
{St Johns University, Queens, NY, USA}%
\email[GYaG]{grabarng@stjohns.edu}%
%\urladdr{http://www.authorone.uni-aone.de}
%\author{Semen}
%\curraddr[A.~Two]{Author Two current address, line 1\newline%
%\indent Author Two current address, line 2}%
%\email[A.~Two]{author-two@authortwo-inst.hu}%
%\urladdr{http://www.authortwo.uni-atwo.hu}
%\address[A. Three]{Author Three address, line 1\newline%
%\indent Author Three address, line 2}
%\email[A.~Three]{author-three@authorthree-inst.edu}%
%\urladdr{http://www.authorthree.uni-athree.edu}
%\thanks{Thanks for Author One.}
%\thanks{Thanks for Author Two.}
%\thanks{This paper is in final form and no version of it will be submitted for
%publication elsewhere.}
\date{March 1, 2016}
%\dedicatory{Dedicated to Professor XY on the occasion of his seventieth birthday.}

\begin{abstract}
In this paper we establish individual ergodic theorem 
for positive kernels (or so called Danford Shwartz (DS+) operators acting on non commutative
symmetric spaces.
\end{abstract}
\maketitle

\section{Introduction}
The goal of the paper is to see that one of the results by Veksler \cite{ve} or Muratov, Pashkova and Rubshtein \cite{rmp1} remains valid for the non commutative case.  
\footnote{
It become known to author that Litvinov and Chilin also saw this result at the same time and wrote it at the same time.
I suggested to them to combine results and names on paper. Waiting for the answer.
}
Let $M, \tau$ be a semifinite von Neumann algebra with semifinite normal faithful trace $\tau$.

In addition we assume that $M,\tau$ satisfy homogeneity property, it may be presented in the form of the resonant property on trace, see for example \cite{bs}.

The space of all measurable operators affiliated with $M, \tau$ in the Sigal \cite{se} sense is denoted by $L_0$, see for details \cite{se, ddp, ddp1}.

Notions of $L_1(M,\tau)$ and $L_\infty(M,\tau)$ was naturally introduced in the same paper.

Since we fix algebra $M$ and trace $\tau$, we omit them from the notations from now on.

Let $F\in M$ be a set of finite linear combinations of orthogonal projections with finite trace. 

Space $R_0=(L_1+L_\infty)_0$ is the closure of $F$ in the norm $\|x\| = inf \{\|x_1\|_1+\|x_2\|_\infty, | x=x_1+x_2,\  x_1\in\L_1,  x_2\in\L_\infty\}  $.

\begin{remark}
The space $R_0$ is not necessary separable,
it is sufficient to consider  $M=B(H)$  of all bounded operators in the Hilbert space $H$ with not separable $H$ with natural trace $\tau$. 
\end{remark}

\begin{definition}
Non-commutative re-arrangement invariant (or symmetric) space $L$ for the fully symmetric case were introduced by Yeadon in \cite{ye2}.  
For the definition of the symmetric spaces we refer to the recent book of Lord, Sukochev, Zanin on singular traces, with original proofs due to Kalton and Sukochev.
\end{definition}

\begin{definition}
Re-arrangement invariant space $L$ over $(M,\tau)$ is called minimal if $F$ is dense in $L$ by norm of $L$.
\end{definition}

\section{Embedding theorem}
The following refinement of the embedding theorem \cite{kps} take place.

\begin{theorem}
\label{th:1}
Let $L$ be a re-arrangement invariant space over $(M,\tau)$.
Then 
\begin{enumerate}
	\item If $L$ is minimal, then 
	$$L_1\cap L_\infty\subseteq L\subseteq R_0$$
	\item If $L$ is not minimal $L$, then
	$$L_\infty\subseteq L\subseteq L_1+L_\infty$$
\end{enumerate}
\end{theorem}

\begin{proof}
Proof is given in the forthcoming paper of author and some co-authors.
\end{proof}

\section{Individual ergodic theorem for minimal symmetric spaces}

\subsection{The Largest Minimal Symmetric Space}
The largest minimal symmetric space is $R_0=(L_1+L_\infty)_0$ is a 
set of all measurable integrable with trace operators plus bounded with not increasing re-arrangement 
functions decreasing to 0. 

The space $R_0$ is minimal symmetric space. 
Any minimal symmetric space is a subset of $R_0$.

\subsection{Positive double contraction on $R_0$}
Any positive kernel ($T\in DS^+$) leaves $R_0$ invariant.  

\subsection{Mean Ergodic Theorem}
Von Neumann Mean ergodic Theorem on $L_2$ follows from the general von Neumann ergodic theorem for the contractions on the general Hilbert spaces.

Mean convergence  on $L_1\cap L_\infty$ follows from the the fact that 
both $L_1$ and  $L_\infty$ are invariant under positive kernels. 
The space $L_1\cap L_\infty$ itself is invariant under action of positive kernel,
and positive kernel is a contraction of the space $L_1\cap L_\infty$. 
The von Neumann Ergodic theorem and the closedness of  $L_1\cap L_\infty$ in its norm implies mean ergodic theorem on $L_1\cap L_\infty$.

\subsection{Mean convergence on $R_0$} 
\begin{proposition}
\label{pro:1}
The Cesaro averages $S_n(T)x$ converge to some $\tilde{x}$ in norm of $R_0$.
\end{proposition}
\begin{proof}
Follows from von Neumann ergodic theorem and the fact that $L_2$ is dense $R_0$.
Details.
We show that $s_l(T)x, l=1,2,...$ is fundamental sequence in $R_0$.
Indeed,  
each $x\in R_0$ may be presented as $x=x_{1,n}+x_{2,n}$ with $x_{1,n}\in L_1$ and $x_{2,n}\in L_\infty$ with $\|x_{2,n} \|_\infty<2^{-n}$.
Then we can apply von Neumann or Yeadon's Mean Ergodic theorem 4.2 \cite{ye2} for $L_1$ to $x_{1,n}$ and find $l(n)$ such that 
for $l,m\geq l(n)$ holds
$$
\|s_l(T)x_{1,n}-s_m(T)x_{1,n}\|_{L_1}<2^{-n}.
$$

Then for $l,m\geq l(n)$
$$
\|s_l(T)x-s_m(T)x\|_{R_0}
\leq \|s_l(T)x_{1,n}-s_m(T)x_{1,n}\|_{L_1}+
$$$$+\|s_l(T)x_{2,n}-s_m(T)x_{2,n}\|_{L_\infty}
\leq 4*2^{-n},
$$
and, hence, the sequence $s_l(T)x$ is fundamental.
Completeness of $R_0$ implies existence of $\tilde{x}\lim_{l\to\infty}s_l(T)x$.  
\end{proof}

\begin{remark}
Mean ergodic theorem for fully symmetric spaces is due to Yeadon  \cite{ye2}, Theorem 4.2.
Note that we do not require the space $L$ to be fully symmetric here.
Condition ii) of the theorem 4.2 \cite{ye2} means that the space $L$ is minimal.
The space $R_0$ does not satisfy condition iii) in theorem 4.2 \cite{ye2}.
\end{remark}

\subsection{Individual Ergodic Theorem in $L_1$}
Individual Ergodic theorem for $L_1$ was established by Yeadon \cite{ye1}, among other authors. 

\subsection{Individual ergodic theorem for $R_0$}
The goal of the section is to show double side almost everywhere convergence for 
the operators from the $R_0$.

\begin{definition}
The sequence $x_n$ from $L_0$ is called converging double side almost  everywhere to $x_0\in L_0$ if
for every $\epsilon>0$ there exist orthogonal projection $E\in M$ such that
$\tau(1-E)<\epsilon$, $E (x_n-x_0 )E\in M$ and $E (x_n-x_0 )E\to 0$.
\end{definition}

\begin{theorem}
Let $M,\tau, R_0$ are as above and $T$ is positive kernel on $M$.
For $x\in R_0$, Cesaro averages $S_n(T)x$ converge d.s.a.e. in  $R_0$.
\end{theorem}
\begin{proof}
The proof follows the line of the proof of Proposition \ref{pro:1} and uses Yeadon's individual ergodic theorem for $L_1$ \cite{ye1}, see also Chilin, Litvinov \cite{cl1}, \cite{cl2}

We show that $s_l(T)x, l=1,2,...$ is fundamental d.s.a.e. sequence in $R_0$.
Indeed,  
each $x\in R_0$ may be presented as $x=x_{1,n}+x_{2,n}$ with $x_{1,n}\in L_1$ and $x_{2,n}\in L_\infty$ with $\|x_{2,n} \|_\infty<2^{-4*n}$.
In turn, $x_{1,n}=x_{1,1,n}+x_{1,2,n}$, with $x_{1,1,n}\in L_\infty$, $x_{1,2,n}\in L_1$ and
$\|x_{1,2,n}\|_{L_1}<2^{-8*n}$, $n=1,2,...$.

Then we can apply Yeadon's Individual Ergodic theorem 1, \cite{ye1} for $L_1$ to $x_{1,1,n}$ and find $l(n)$ and projector $E(n)\in M$ such that $\tau(1-E(n))<2^{-4*n}$ and
for $l,m\geq l(n)$ holds
$$
\|E(n)(s_l(T)x_{1,1,n}-s_m(T)x_{1,1,n})E(n)\|_{L_\infty}\}\to 0. 
$$
 
We can represent $x_{1,2,n} =\sum_{k=1}^\infty x_{2,n,k} $, with $x_{2,n,k}\in L_\infty$ and $\|x_{2,n,k}\|_{L_1}<2^{-8*(n+k)}$.

Then, we can find $E(1,n)=\wedge_k E(1,n,k)$, with $\tau(1- E(1,n))\leq 2^{-4*n}$ and
$E(1,n)s_l(x_{1, 2,n})E(1,n)\in L_\infty$ and $\|E(1,n)s_l(x_{2,n})E(1,n)\|_{ L_\infty}<2^{-4*n}$, 
where projections $E(1,n,k)$ are obtain by Theorem  1 from \cite{ye1} applied to $x_{2,n,k}$ and $\epsilon=2^{-4(n+k)}$.

By choosing $E(2,n)=\wedge_{k=1}^\infty E(1,n+k) \wedge \wedge_{k=1}^\infty E(n+k)$, we have
$\tau(1-E(2,n))<2^{-4*n}$.

Moreover, for $l,m\geq l(n)$
$$
\|E(2,n)(s_l(T)x-s_m(T)x)E(2,n)\|_{L_\infty}
$$$$
\leq \|E(2,n)(s_l(T)x_{1,n}-s_m(T)x_{1,n})E(2,n)\|_{L_\infty}+\|E(2,n)(s_l(T)x_{2,n}-s_m(T)x_{2,n})E(2,n)\|_{L_\infty}
$$$$
\leq \|E(2,n)(s_l(T)x_{1,1,n}-s_m(T)x_{1,1,n})E(2,n)\|_{L_\infty}+\|E(2,n)(s_l(T)x_{1,2,n}E(2,n)\|_{L_\infty}
$$$$
+\|E(2,n)s_m(T)x_{1,2,n}E(2,n)\|_{L_\infty}+2*2^{-4*n} \leq 8*2^{-n},
$$
and, hence, the sequence $s_l(T)x$ is fundamental d.s.a.e.
The sequence $s_l(T)x$ also converges in norm in $R_0$ to $\tilde{x}$, hence it converges in measure.
This implies convergence of $s_l(T)x$ d.s.a.e. to $\tilde{x}$.
\end{proof}

\begin{remark}
In the case when space $L$ is not minimal, it contains $M$.
Then it is possible to show \cite{rmp1}, that even in the commutative case, 
there exists ergodic automorphism of the space with measure such that 
ergodic averages do not converge almost everywhere.
\end{remark}

\begin{corollary}
Let $L$ be a minimal non-commutative symmetric space.
Let $T$ be a positive kernel such that $T$ leaves $L$ invariant and $T$ acts as contraction on $L$.
Then Cesaro averages $s_n(T)x$ converge in norm and d.s.a.e. for any $x\in L$.
\end{corollary}
\begin{proof}
Since the space $L$ is minimal, the set of $L_1\cap L_\infty$ is dense in $L$. 
Since $L_1\cap L_\infty\subseteq L$, hence 
$$\|x\|_{L_1\cap L_\infty}\geq C* \|x\|_L$$
for any $x\in L_1\cap L_\infty$,
which in turn implies convergence of Cesaro averages of   $s_n(T)x$ in norm of $L$
for $x\in  L_1\cap L_\infty$.
Fix real $\epsilon>0$. Find $x_k\in  L_1\cap L_\infty$ with $\|x-x_k\|_L<\epsilon/2$. Then Cesaro averages are 
$s_n(T)x$ are within $\epsilon$ of the $\tilde{x_k}$, where $\tilde{x_k}=\lim_{n\to\infty} s_n(T)x_k$.
This implies norm convergence of   $s_n(T)x$.

The d.s.a.e. convergence follows from the embedding theorem \ref{th:1}, since a minimal re-arrangement invariant non-commutative function space $L$ is a subspace of $R_0$.
\end{proof}

\begin{corollary}
\label{col:1}
Let $L$ be a minimal fully symmetric space.
Let $T$ be a positive kernel on $(M,\tau)$. Then $T$ leaves $L$ invariant and act on $L$ as contraction. 
Moreover,  the Cesaro averages $S_n(T)x$ converge d.a.e. for any $x\in L$.
\end{corollary}

\begin{corollary} (see Chilin Litvinov \cite{cs2}).
 Let $L_\Psi$ be a non-commutative Orlicz space with function $\Psi$ satisfying conditions $\delta_2$ and $\Delta_2$.
Let $T$ be a positive kernel. Then the Cesaro averages $S_n(T)x$ converge d.s.a.e. for any $x\in L$.
\end{corollary}
\begin{proof}
The Orlicz space  $L_\Psi$ with function $\Psi$ satisfying conditions $\delta_2$ and $\Delta_2$ is minimal \cite{bs, kps, ssmf}.
Since the space $L_\Psi$ is fully symmetric, it is interpolation space \cite{bs, kps}
 and hence $T$ leaves $L_\Psi$ invariant and acts on $L_\Psi$ as a contraction. 
Then we are in the assumptions of the  Corollary \ref{col:1}, and hence  $S_n(T)x$ converges d.a.e. .
\end{proof}

%\subsection{Case of Besicovich sequences}

%\subsection{ }

%\subsection{Convergence of martingales}
%In this section we formulate 

\end{document}